\documentclass[11pt,article]{amsart} 
\usepackage{amscd,amsmath,amssymb,amsthm,amsfonts,epsfig,graphics}
\usepackage{amssymb,latexsym}
\usepackage{cite} 

\usepackage[height=190mm,width=130mm]{geometry} 

\usepackage{amsthm}
\theoremstyle{plain}
\newtheorem{theorem}{Theorem}
\newtheorem{lemma}{Lemma}

\theoremstyle{definition}

\theoremstyle{remark}
\newtheorem{remark}{Remark}

\numberwithin{equation}{section} 

\begin{document}
\title[On the Diophantine equation $(x+1)^{k}+(x+2)^{k}+...+(lx)^{k}=y^{n}$]{On the Diophantine equation $(x+1)^{k}+(x+2)^{k}+...+(lx)^{k}=y^{n}$} 

\author{G\"{o}khan Soydan}

\address{Department of Mathematics, Uluda\u{g} University, 16059 Bursa, Turkey}

\email{gsoydan@uludag.edu.tr}

\thanks{}

\begin{abstract}
Let $k,l\geq2$ be fixed integers. In this paper, firstly, we prove that all solutions of the equation $(x+1)^{k}+(x+2)^{k}+...+(lx)^{k}=y^{n}$ in integers $x,y,n$ with $x,y\geq1, n\geq2$ satisfy $n<C_{1}$ where $C_{1}=C_{1}(l,k)$ is an effectively computable constant. Secondly, we prove that all solutions of this equation in integers $x,y,n$ with $x,y\geq1, n\geq2, k\neq3$ and $l\equiv0 \pmod 2$ satisfy $\max\{x,y,n\}<C_{2}$ where $C_{2}$ is an effectively computable constant depending only on $k$ and $l$.
\end{abstract}

\dedicatory{To my wife and my daughter}

\date{\today}

\maketitle
\section{Introduction}

In 1956, J.J. Sch{\"a}ffer \cite{Sch} considered the equation

\begin{align}\label{eq.1.1}
1^{k}+2^{k}+...+x^{k}=y^{n}.
\end{align}
He proved that for fixed $k\geq1$ and $n\geq2$, (\ref{eq.1.1}) has at most finitely many solutions in positive integers $x$ and $y$, unless
\begin{align*}
(k,n)\in\{(1,2),(3,2),(3,4),(5,2)\},
\end{align*} 
where, in each case, there are infinitely many such solutions.

Sch{\"a}ffer's proof used an ineffective method due to Thue and Siegel so his result is also ineffective. This means that the proof does not provide any algorithm to find all solutions. Applying Baker's method, K. Gy\H{o}ry, R. Tijdeman and M. Voorhoeve \cite{GTV} proved a more general and effective result in which the exponent $n$ is also unknown.

Let $k\geq2$ and $r$ be fixed integers with $k\notin\{3,5\}$ if $r=0$, and let $s$ be a square-free odd integer. In \cite{GTV}, they proved that the equation
\begin{align*}
s(1^{k}+2^{k}+...+x^{k})+r=y^{n}
\end{align*}
in positive integers $x,y\geq2$, $n\geq2$ has only finitely many solutions and all these can be effectively determined.
Of particular importance is the special case when $s=1$ and $r=0$. They also showed that for given $k\geq2$ with $k\notin\{3,5\}$, equation (\ref{eq.1.1}) has only finitely many solutions in integers $x,y\geq1$, $n\geq2$, and all these can be effectively determined.
The following striking result is due to Voorhoeve, Gy\H{o}ry and Tijdeman \cite{VGT}.

Let $R(x)$ be a fixed polynomial with integer coefficients and let $k\geq2$ be a fixed integer such that $k\notin\{3,5\}$. In  \cite{VGT}, same authors proved that the equation
\begin{align*}
1^{k}+2^{k}+...+x^{k}+R(x)=by^{n}
\end{align*}
in integers $x,y\geq2$, $n\geq2$ has only finitely many solutions, and an effective upper bound can be given for $n$.  
Later, various generalizations and analogues of the results of Gy\H{o}ry, Tijdeman and Voorhoeve have been established by several authors \cite{Br1}, \cite{Br2}, \cite{BP1}, \cite{BP2}, \cite{Di},\cite{Ka}, \cite{Pi}, \cite{Ur}. For a survey of these results we refer to \cite{GP} and the references given there. 

Here we present the result of B. Brindza \cite{Br2}. For brevity let us set $S_{k}(x)=1^{k}+2^{k}+...+x^{k}$, $A=\mathbb{Z}[x]$, $\kappa=(k+1)\displaystyle{\prod_{(p-1)|(k+1)!}}p$ ($p$ prime). Let
\begin{align*}
F(y)=Q_{n}y^{n}+...+Q_{1}y+Q_{0}\in A[y].
\end{align*}
Consider the equation
\begin{align}\label{eq.1.2} 
F(S_{k}(x))=y^{n}
\end{align}
in integers $x,y\geq2$, $n\geq2$. Let $Q_{i}(x)=\kappa^{i}K_{i}(x)$ where $K_{i}(x)\in \mathbb{Z}[x]$ for $i=2,3,...,m$. In \cite{Br2}, Brindza proved that if $Q_{i}(x)\equiv0\pmod{\kappa^{i}}$, for $i=2,3,...,m$; $Q_{1}(x)\equiv\pm1\pmod{4}$ and $k\notin\{1,2,3,5\}$, then all solutions of \eqref{eq.1.2} satisfy $\max\{x,y,n\}<C_{1}$, where $C_{1}$ is an effectively computable constant depending only on $F$ and $k$.

Recently C. Rakaczki \cite{Ra} gave a generalization of the results of Gy\H{o}ry, Tijdeman and Voorhoeve and an extension of the result of Brindza to the case when the polynomials $Q_{i}(x)$ are arbitrary constant polynomials.

Let $F(x)$ be a polynomial with rational coefficients and $d\neq0$ be an integer. Suppose that $F(x)$ is not an $n$-th power. In \cite{Ra}, Rakaczki showed that the equation 
\begin{align*}
F(S_{k}(x))=dy^{n}
\end{align*}
has only finitely many integer solutions $x,y\geq2$, $n\geq2$, which can be effectively determined provided that $k\geq6$. 

Let $k>1$, $r,s\neq0$ be fixed integers. Then apart from the cases when $(i)$ $k=3$ and either $r=0$ or $s+64r=0$, and $(ii)$ $k=5$ and either $r=0$ or $s-324r=0$, Rakaczki proved that the equation
\begin{align*}
s(1^{k}+2^{k}+...+x^{k})+r=y^{n}
\end{align*}
in integers $x>0$, $y$ with $|y|\geq2$, and $n\geq2$ has only finitely many solutions which can be effectively determined. 

Recently, Z. Zhang \cite{Zh} studied the Diophantine equation
\begin{align*}
(x-1)^{k}+x^{k}+(x+1)^{k}=y^{n}, n>1,
\end{align*} 
and completely solved it for $k=2,3,4$. Now we consider a more general equation. Let
\begin{align*}
G(x)=(x+1)^{k}+(x+2)^{k}+...+(lx)^{k}.
\end{align*}

In this paper, we are interested in the solutions of the equation
\begin{align}\label{eq.1.3}
G(x)=y^{n}
\end{align}
in integers $x,y\geq1$ and $n\geq2$.
\setcounter{theorem}{0}
\begin{theorem} \label{theo.1}
Let $k,l\geq2$ fixed integers. Then all solutions of the equation \eqref{eq.1.3} in integers $x,y\geq1$ and $n\geq2$ satisfy $n<C_{1}$ where $C_{1}$ is an effectively computable constant depending only on $l$ and $k$.  
\end{theorem}
\begin{theorem}\label{theo.2}
Let $k,l\geq2$ fixed integers such that $k\neq3$. Then all solutions of the equation \eqref{eq.1.3} in integers $x,y,n$ with $x,y\geq1$, $n\geq2$, and $l\equiv0 \pmod 2$ satisfy $\max\{x,y,n\}<C_{2}$ where $C_{2}$ is an effectively computable constant depending only on $l$ and $k$.
\end{theorem} 

We organize this paper as follows. In Section 2, firstly, we recall the general results that we will need. Secondly, we give two new lemmas and prove that these lemmas imply our theorems. In Section 3, we discuss the number of solutions in integers $x,y\geq 1$ of \eqref{eq.1.3} where $n>1$ is fixed $k\in\{1,3\}$ and $l\equiv0 \pmod 2$ and reformulate this case. In the last section, we give the proofs of Theorems \ref{theo.1} and \ref{theo.2}.

\section{Auxiliary Results}
\begin{lemma} \label{lem.1}
$(x+1)^{k}+(x+2)^{k}+...+(lx)^{k}=\dfrac{B_{k+1}(lx+1)-B_{k+1}(x+1)}{k+1}$ where
\begin{align*}
B_{q}(x)=x^{q}-\frac{1}{2}qx^{q-1}+\dfrac{1}{6}\binom {q} {2}x^{q-2}+...=\sum\limits_{i=0}^q\binom {q} {i}B_{i}x^{q-i}
\end{align*}
is the q-th Bernoulli polynomial with $q=k+1$.
\end{lemma}
\begin{proof}
It is an application of the equality
\begin{align*}
\sum\limits_{n=M}^{N-1} n^{k}=\frac{1}{k+1} (B_{k+1}(N)-B_{k+1}(M))
\end{align*}
which is given by Rademacher in \cite{Rd}, pp.3-4.
\end{proof}

Now we give an important result of Brindza which is an effective version of Leveque's theorem \cite{Le}
\begin{lemma}[Brindza]\label{lem.2}
Let $H(x)\in\mathbb{Q}[x]$,
\begin{align*}
H(x)=a_{0}x^{N}+...+a_{N}=a_{0}\prod_{i=1}^n (x-\alpha_{i})^{r_{i}},
\end{align*}
with $a_{0}\neq0$ and $\alpha_{i}\neq\alpha_{j}$ for $i\neq j$. Let $0\neq b\in\mathbb{Z}$, $2\leq m\in\mathbb{Z}$ and define $t_{i}=\frac{m}{(m,r_{i})}$. Suppose that $\{t_{1},...,t_{n}\}$ is not a permutation of the n-tuples

$(a)$ $\{t,1,...,1\}$, $t\geq1$; $(b)$ $\{2,2,1,...,1\}.$

Then all solutions $(x,y)\in\mathbb{Z}^{2}$ of the equation
\begin{align*}
H(x)=by^{m}
\end{align*}
satisfy  $\max\{|x|,|y|\}<C$, where $C$ is an effectively computable constant depending only on $H$, $b$ and $m$.
\end{lemma}
\begin{proof}
See B. Brindza \cite{Br2}.
\end{proof}
\begin{lemma}[Schinzel \& Tijdeman] \label{lem.3}
Let $0\neq b \in\mathbb{Z}$ and let $P(x) \in\mathbb{Q}[x]$ be a polynomial with at least two distinct zeroes. Then the equation
\begin{align*}
P(x)=by^{n}
\end{align*}
in integers $x,y>1$, $n$ implies that $n<C$ where $C=(P,b)$ is an effectively computable constant.
\end{lemma}
\begin{proof}
See A. Schinzel and R. Tijdeman \cite{ST}.
\end{proof}
\begin{lemma}\label{lem.4}
For $k \in\mathbb{Z}^{+}$ let $B_{k}(x)$ be the k-th Bernoulli polynomial. Then the polynomial
\begin{align*}
G(x)=\frac{B_{k+1}(lx+1)-B_{k+1}(x+1)}{k+1}
\end{align*}
has at least two distinct zeroes.
\end{lemma}
\begin{proof}
By Lemma \ref{lem.1}, we have $G(x)=(\frac{l^{k+1}-1}{k+1})x^{k+1}+(\frac{l^{k}-1}{2})x^{k}+...+cx$
where $c$ is a rational number. Now one can observe that the coefficient of $x^{k}$ is nonzero and that $x=0$ is a zero of $G(x)$. Let's also assume that there is no other zero of $G(x)$. Thus we have
\begin{align*}
G(x)=\bigg(\frac{l^{k+1}-1}{k+1}\bigg)x^{k+1}
\end{align*}
which is a contradiction.  
\end{proof}
\begin{lemma}[Voorhoeve, Gy\H{o}ry and Tijdeman] \label{lem.5}
Let $q\geq2$, $R^{*}(x)\in\mathbb{Z}[x]$ and set
\begin{align*}
Q(x)=B_{q}(x)-B_{q}+qR^{*}(x).
\end{align*}
Then

$(i)$ $Q(x)$ has at least three zeros of odd multiplicity, unless $q\in\{2,4,6\}$.

$(ii)$ For any odd prime $p$, at least two zeros of $Q(x)$ have multiplicities relatively prime to $p$.
\end{lemma}
\begin{proof}
See M. Voorhoeve, K. Gy\H{o}ry and R. Tijdeman \cite{VGT}. 
\end{proof}
\begin{lemma}\label{lem.6}
For $q\geq2$ let $B_{q}(x)$ be the q-th Bernoulli polynomial. Let  
\begin{align}\label{eq.1.4}
P(x)=B_{q}(lx+1)-B_{q}(x+1)
\end{align}
where $l$ is even. Then

$(i)$ $P(x)$ has at least three zeros of odd multiplicity unless $q\in\{2,4\}$.

$(ii)$ For any odd prime $p$, at least two zeros of $P(x)$ have multiplicities relatively prime to $p$.
\end{lemma}
\begin{proof}
We shall follow the proof of Lemma \ref{lem.5} of \cite{VGT}. By the Staudt-Clausen theorem (see Rademacher \cite{Rd}, pp.10),  the denominators of the Bernoulli numbers $B_{i}$, $B_{2k}$ $(k=1,2,...)$ are even but not divisible by $4$. Choose the minimal $d\in\mathbb{N}$ such that both the polynomials $d(B_{q}(lx+1)-B_{q}(x+1))$ and $dB_{q}(x)$ are in $\mathbb{Z}[x]$. Using the equality $B_{q}(x+1)=B_{q}(x)+qx^{q-1}$ (see \cite{Rd}, pp.4-5), we have
\begin{align}\label{eq.1.5}
dP(x)=d\left( \sum\limits_{i=0}^q\binom {q} {i}\left[ (lx+1)^{q-i}-x^{q-i}\right] B_{i}-qx^{q-1}\right).
\end{align}
Hence by the choice of $d$ and by the Staudt-Clausen theorem, we have $d\binom {q} {i}B_{i}\in\mathbb{Z}$ and $\binom {q} {2k}dB_{2k}\in\mathbb{Z}$ for $k=1,2,...,\frac{q-1}{2}$. If $d$ is odd, then necessarily $\binom {q} {i}$ and $\binom {q} {2k}$ must be even for $k=1,2,...,\frac{q-1}{2}$. Write $q=2^{\mu}r$ where $\mu\geq1$ and $r$ is odd. Then $\binom {q} {2^{\mu}}$ is odd, giving a contradiction unless $r=1$. So
\begin{align*}
\textrm{$d$ is odd} \Longleftrightarrow q=2^{\mu} \textrm{ for some } \mu\geq1.
\end{align*}
If $q\neq2^{\mu}$ for any $\mu\geq1$ then
\begin{align}\label{eq.1.6}
d\equiv2\pmod{4}.
\end{align}
We distinguish three cases:\vspace{3mm}  

\textsc{I. }Suppose $q=2^{\mu}$ for some $\mu\geq1$, so that $d$ is odd. We first prove $(i)$ so we may assume that $\mu\geq3$. Considering \eqref{eq.1.5} modulo 4, we have 
\begin{align}\label{eq.1.7}
dP(x)&\equiv d\sum\limits_{i=0}^{q-2}\binom {q} {i}(lx+1)^{q-i}B_{i}-d \sum\limits_{i=0}^{\frac{q-2}{2}}\binom {q} {2i}B_{2i}x^{q-2i}\pmod{4}.
\end{align}
Firstly, let $l\equiv 0 \pmod{4}$. Then we obtain
\begin{align}\label{eq.1.8}
d\sum\limits_{i=0}^{q-2}\binom {q} {i}(lx+1)^{q-i}B_{i}\equiv d\sum\limits_{i=0}^{q-2}\binom {q} {i}B_{i} \equiv d\sum\limits_{i=0}^{\frac{q-2}{2}}\binom {q} {2i}B_{2i}\pmod{4}.
\end{align}
It is easy to see that $\sum\limits_{i=1}^{q}\binom {q} {q-i}B_{q-i}=0$. Hence we get
\begin{align}\label{eq.1.9}
\sum\limits_{i=1}^{\frac{q-2}{2}}\binom {q} {2i}B_{2i}=-B_{0}-qB_{1}.
\end{align}
By using \eqref{eq.1.8} and \eqref{eq.1.9}, one gets
\begin{align*}
d\sum\limits_{i=0}^{q-2}\binom {q} {i}(lx+1)^{q-i}B_{i}\equiv d \bigg(\binom {q} {0}B_{0}+\sum\limits_{i=1}^{\frac{q-2}{2}}\binom {q} {2i}B_{2i}\bigg) \equiv0\pmod{4}.
\end{align*}
Then we deduce by \eqref{eq.1.7} the following:
\begin{align}\label{eq.1.10}
dP(x)\equiv-d\sum\limits_{i=0}^{\frac{q-2}{2}}\binom {q} {2i}B_{2i}x^{q-2i}\pmod{4}.
\end{align}
Secondly, let $l\equiv 2 \pmod{4}$. Then we obtain
\begin{align}\label{eq.1.11}
d\sum\limits_{i=0}^{q-2}\binom {q} {i}(lx+1)^{q-i}B_{i}\equiv d\sum\limits_{i=0}^{q-2}\binom {q} {i}(2x+1)^{q-i}B_{i}\pmod{4}.
\end{align}
Then the RHS of \eqref{eq.1.11} becomes
\begin{equation} \label{eq.1.12}
\begin{aligned}
d\sum\limits_{i=0}^{q-2}\binom {q} {i}(2x+1)^{q-i}B_{i}=d(B_{0}.(2x+1)^{q}+qB_{1}.(2x+1)^{q-1}\\+\sum\limits_{i=0}^{\frac{q-2}{2}}\binom {q} {2i}(2x+1)^{q-2i}B_{2i}.
\end{aligned}
\end{equation}
Since $2x+1$ is odd and $q=2^{\mu}, \mu\geq3,$ is even, considering \eqref{eq.1.12} modulo 4 and using \eqref{eq.1.9}, \eqref{eq.1.11} becomes
\begin{align*}
d\sum\limits_{i=0}^{q-2}\binom {q} {i}(lx+1)^{q-i}B_{i}\equiv0\pmod{4}.
\end{align*}
So in all cases \eqref{eq.1.7} reduces to \eqref{eq.1.10}.

Note that $\binom {q} {2i}$ is divisible by $8$ unless $2i$ is divisible by $2^{\mu-2}$. We have therefore for some odd $d'$, writing $t=\frac{1}{4}q$
\begin{align}\label{eq.1.13}
dP(x)\equiv d'x^{4t}+2x^{3t}+dx^{2t}+2x^{t} \pmod{4}.
\end{align}
Write $dP(x)=R^{2}(x)S(x)$ where $R(x),S(x)\in\mathbb{Z}[x]$ and $S$ contains each factor of odd multiplicity of $P$ in $\mathbb{Z}[x]$ exactly once. Assume that deg$S(x)\leq2$. Since
\begin{align*}
R^{2}(x)S(x)\equiv x^{4t}+x^{2t} \equiv x^{2t}(x^{2t}+1)\pmod{2}, 
\end{align*}  
$R^{2}(x)$ must be divisible by $x^{2t-2}\pmod{2}$. So
\begin{align*}
R(x)=x^{t-1}R_{1}(x)+2R_{2}(x),
\end{align*}
\begin{align*}
R^{2}(x)=x^{2t-2}R_{1}^{2}(x)+4R_{3}(x),
\end{align*}
for certain $R_{1},R_{2},R_{3}\in\mathbb{Z}[x]$. If $q>8$, then $t>2$ so the last identity is incompatible with \eqref{eq.1.13} because of the term $2x^{t}$. Hence deg$S(x)\geq3$, which proves $(i)$. If $q=8$, then by \eqref{eq.1.10}
\begin{align*}
dP(x)\equiv 3x^{8}+2x^{6}+x^{4}+2x^{2}\pmod{4}.
\end{align*} 
From here, we follow the proof in the corrigendum paper \cite{VGT}. This fact can also be reduced from \eqref{eq.1.10}. So, the proof of $(i)$ is completed where $q=2^{\mu}$, $\mu\geq3$.

To prove $(ii)$, let $p$ be an odd prime and write $P(x)=(R(x))^{p}S(x)$ where $R,S\in\mathbb{Z}[x]$ and all the roots of multiplicity divisibly by $p$ are incorporated in $(R(x))^{p}$. We have, writing $\delta=\frac{1}{2}q$, by \eqref{eq.1.13}
\begin{align*}
dP(x)\equiv (R(x))^{p}S(x)\equiv x^{\delta}(x^{\delta}+1)\equiv x^{\delta}(x+1)^{\delta}  \pmod{2}.
\end{align*}
Since $\delta$ is prime to $p$, $S$ has at least two different zeros, proving $(ii)$ in case \textsc{I}.\vspace{3mm} 

\textsc{II. }Suppose $q$ is even and $q\neq2^{\mu}$ for any $\mu$. Then $d\equiv 2\pmod{4}$ and hence considering \eqref{eq.1.5} in modulo $2$, we get
\begin{align*}
dP(x)\equiv d\sum\limits_{i=0}^q\binom {q} {i}(1-x^{q-i})B_{i}\pmod{2}.
\end{align*} 
Since $B_{i}d\binom {q} {i}\equiv\binom {q} {i}\pmod{2}$ for $i=1,2,3,...,q$, we have
\begin{align*}
dP(x)\equiv \sum\limits_{k=1}^{\frac{q-2}{2}}\binom {q} {2k}x^{2k}=\sum\limits_{t=1}^{q-1}\binom {q} {t}x^{t}\equiv (x+1)^{q}-x^{q}-1 \pmod{2}.
\end{align*}
Write $q=2^{\mu}r$, where $r>1$ is odd. Then
\begin{align*}
dP(x)\equiv (x+1)^{q}-x^{q}-1\equiv ((x+1)^{r}-x^{r}-1)^{{2}^{\mu}}\pmod{2}.
\end{align*}
Since $r>1$ is odd, $(x+1)^{r}-x^{r}-1$ has $x$ and $x+1$ as simple factors $\pmod{2}$. Thus
\begin{align*}
dP(x)\equiv x^{2^{\mu}}(x+1)^{{2}^{\mu}}K(x) \pmod{2}
\end{align*} 
where $K(x)$ is neither divisible by $x$ nor by $(x+1)$ $\pmod{2}$. As in the preceding case, $P(x)$ must have two roots of multiplicity prime to $p$. This proves part $(ii)$ of the lemma.

In order to prove part $(i)$, first we consider the case $q=6$. In this case
\begin{align*}
dP(x)\equiv (2l^{6}+2)x^{6}+(2l^{5}+2)x^{5}+(l^{4}+3)x^{4}+(3l^{2}+1)x^{2}\pmod{4}.
\end{align*}
Since $l$ is even, we can write
\begin{align*}
dP(x)\equiv 2x^{6}+2x^{5}+3x^{4}+x^{2}\pmod{4}.
\end{align*}
So, $P(x)$ has at least three simple roots.
To prove our claim, suppose $dP$ can be written as
\begin{align}\label{eq.1.14}
dP(x)\equiv S(x)R^{2}(x) \pmod{4}
\end{align}
with deg$S\leq2$.

If deg$S=0$, then clearly $S$ is an odd constant, so $R^{2}(x)\equiv x^{4}+x^{2}\pmod{2}$. Hence $R(x)\equiv x^{2}+x\pmod{2}$ and $R^{2}(x)\equiv x^{4}+2x^{3}+x^{2}\pmod{4}$, which is a contradiction. If deg$S=1$, then either $S(x)\equiv x$ or $S(x)\equiv x+1\pmod{2}.$ In both cases, the quotient of $P$ and $S$ can not be written as a square $\pmod{2}$. If deg$S=2$, then either $S(x)\equiv x^{2}$ or $S(x)\equiv x^{2}+x$ or $S(x)\equiv x^{2}+1\pmod{2}$,
\\since $x^{2}+x+1$ does not divide $P \pmod{2}$. In the first case $R(x)\equiv x+1\pmod{2}$, hence $R^{2}(x)\equiv x^{2}+2x+1\pmod{4}$ which is a contradiction. In the second case, the quotient of $P$ and $S$ is not even a square$\pmod{2}$. In the third case $R(x)\equiv x\pmod{2}$, hence $R^{2}(x)\equiv x^{2}\pmod{4}$ which is a contradiction. We conclude that $dP$ cannot be written in form \eqref{eq.1.14} with deg$S<3$, proving our claim.

Secondly, as $q=2$ and $4$ are the exceptional cases, $q=6$ case is treated in this section and finally the case $q=8$ was treated in Section I, we may assume that $q\geq10$. Considering \eqref{eq.1.5} modulo $4$ where $d\equiv2\pmod{4}$, we have
\begin{align}\label{eq.1.15}
dP(x)&\equiv d\sum\limits_{i=0}^{q-2}\binom {q} {i}(lx+1)^{q-i}B_{i}-dqB_{1}x^{q-1}-d \sum\limits_{i=0}^{\frac{q-2}{2}}\binom {q} {2i}B_{2i}x^{q-2i}\pmod{4}.
\end{align}
Firstly, let $l\equiv 0 \pmod{4}$. Then we obtain
\begin{align}\label{eq.1.16}
d\sum\limits_{i=0}^{q-2}\binom {q} {i}(lx+1)^{q-i}B_{i}\equiv dqB_{1}+ d\sum\limits_{i=0}^{\frac{q-2}{2}}\binom {q} {2i}B_{2i}\pmod{4}.
\end{align}
We know that $\sum\limits_{i=1}^{q}\binom {q} {q-i}B_{q-i}=0$. By \eqref{eq.1.9} and \eqref{eq.1.16}, one gets
\begin{align*}
d\sum\limits_{i=0}^{q-2}\binom {q} {i}(lx+1)^{q-i}B_{i}\equiv d \bigg(q B_{1}+\sum\limits_{i=0}^{\frac{q-2}{2}}\binom {q} {2i}B_{2i}\bigg) \equiv0\pmod{4}.
\end{align*}
Then we deduce by \eqref{eq.1.15} the following:
\begin{align}\label{eq.1.17}
dP(x)\equiv-dqB_{1}x^{q-1}-d\sum\limits_{i=0}^{\frac{q-2}{2}}\binom {q} {2i}B_{2i}x^{q-2i}\pmod{4}.
\end{align}
Secondly, let $l\equiv 2 \pmod{4}$. Then we have \eqref{eq.1.11} and the RHS of \eqref{eq.1.11} becomes
\begin{equation} \label{eq.1.18}
\begin{aligned}
d\sum\limits_{i=0}^{q-2}\binom {q} {i}(2x+1)^{q-i}B_{i}=d(B_{0}.(2x+1)^{q}+qB_{1}.(2x+1)^{q-1}\\+\sum\limits_{i=1}^{\frac{q-2}{2}}\binom {q} {2i}(2x+1)^{q-2i}B_{2i}.
\end{aligned}
\end{equation}
Since $2x+1$ is odd and $q\neq2^{\mu}$ is even $(q\geq10)$ and $dq\equiv0\pmod{4}$, considering \eqref{eq.1.18} modulo 4 and using \eqref{eq.1.9}, \eqref{eq.1.18} becomes
\begin{align*}
d\sum\limits_{i=0}^{q-2}\binom {q} {i}(lx+1)^{q-i}B_{i}\equiv0\pmod{4}.
\end{align*}
So in all cases \eqref{eq.1.15} reduces to \eqref{eq.1.17}. Then by \eqref{eq.1.17} we have
\begin{align}\label{eq.1.19}
dP(x)\equiv 2x^{q}-qx^{q-1}+\frac{1}{6}d\binom {q} {2}x^{q-2}+...+dB_{q-2}\binom {q} {2}x^{2}\pmod{4}.
\end{align}
Write $dP(x)\equiv R^{2}(x)S(x)$, where $R,S\in\mathbb{Z}[x]$ and $S(x)$ only contains each factor of odd multiplicity of $P$ once. Then deg$S(x)\geq3$. The assertion easily follows by repeating the corresponding part of the proof of Lemma \ref{lem.5}. Thus, the proof is completed for the case \textsc{II}.\vspace{3mm} 

\textsc{III. }  Let $q\geq3$ be odd. Then $d\equiv2\pmod{4}$ and for $i=1,2,4,...,q-1,$
\begin{align*}
d\binom {q} {i}B_{i}\equiv \binom {q} {i} \pmod{2}.
\end{align*} 
Now considering \eqref{eq.1.5} modulo $2$, we have
\begin{align*}
dP(x)\equiv d\sum\limits_{i=0}^{q}\binom {q} {i}(1-x^{q-i})B_{i}\pmod{2}.
\end{align*}
Since $\sum\limits_{\lambda=1}^{\frac{q-2}{2}}\binom {q} {2\lambda}=2^{q-1}-1\equiv 1 \pmod{2}$, we have
\begin{align}\label{eq.1.20}
dP(x)\equiv x^{q-1}+\sum\limits_{\lambda=1}^{\frac{q-1}{2}}\binom {q} {2\lambda}x^{q-2\lambda}\pmod{2}.
\end{align}
From \eqref{eq.1.5}, we get
\begin{align}\label{eq.1.21}
dP'(x)=d(\sum\limits_{i=0}^{q}\binom {q} {i}[(lx+1)^{q-i}-x^{q-i}]B_{i})'-dq(q-1)x^{q-2}
\end{align}
and then
\begin{align}\label{eq.1.22}
xdP'(x)\equiv \sum\limits_{\lambda=1}^{\frac{q-1}{2}}\binom {q} {2\lambda}(q-2\lambda)x^{q-2\lambda}\pmod{2}.
\end{align}
Hence by using \eqref{eq.1.20} and \eqref{eq.1.22}
\begin{align*}
d(P(x)+xP'(x))\equiv x^{q-1}\pmod{2}.
\end{align*}

Any common factor of $dP(x)$ and $dP'(x)$ must therefore be congruent to a power of $x \pmod{2}$. Considering \eqref{eq.1.21} modulo $2$, $dP'(x)\equiv \binom {q} {q-1}=q \equiv 1 \pmod{2}$. Since $dP'(0)\equiv 1 \pmod{2}$, we find that $dP(x)$ and $dP'(x)$ are relatively prime$\pmod{2}$. So any common divisor of $dP(x)$ and $dP'(x)$ in $\mathbb{Z}[x]$ is of the shape $2R(x)+1$. Write $dP(x)=Q(x)S(x)$ where $Q(x)=\displaystyle{\prod_{i}}Q_{i}(x)^{k_{i}}\in\mathbb{Z}[x]$ contains the multiple factors of $dP$ and $S\in\mathbb{Z}[x]$ contains its simple factors where $k_{i}$ denotes the multiplicity of the polynomial factor $Q_{i}(x).$ Then $Q(x)$ is of the shape $2R(x)+1$ with $R\in\mathbb{Z}[x]$, so
\begin{align*}
S(x)\equiv dP(x) \equiv x^{q-1}+...\pmod{2}.
\end{align*} 
Thus the degree of $S(x)$ is at least $q-1$, proving case III whence $q>3$.

If $q=3$, then
\begin{align}\label{eq.1.23}
dP(x)=(l-1)x(2(l^{2}+l+1)x^{2}+3(l+1)x+1).
\end{align}
Considering \eqref{eq.1.23} where $l\equiv 2\pmod{4}$, it follows that
\begin{align*}
dP(x)\equiv x(2x+1)(3x+1)\pmod{4}.
\end{align*}
So, $P(x)$ has three simple roots if $l \equiv 2 \pmod{4}$. Now, we consider the case $l \equiv 0 \pmod{4}$ in \eqref{eq.1.22}. Then we have
\begin{align*}
2P(x)\equiv 2x^{3}+x^{2}+3x \pmod{4}. 
\end{align*}
$P(x)$ has also three simple roots if $l \equiv 0 \pmod{4}$. To prove this, suppose
\begin{align}\label{eq.1.24}
2P(x)\equiv Q(x)T^{2}(x)\pmod{4} 
\end{align}
with deg$Q\leq2$.

If deg$Q=0$, then $Q$ is an odd constant. So the quotient of $2P$ and $Q$ can not be written as a square $\pmod{2}$. If deg$Q=1$, then either $Q(x)\equiv x$ or $Q(x)\equiv x+1 \pmod{2}$. In both case, the quotient of $2P$ and $Q$ can not be written as a square $\pmod{2}$. If deg$Q=2$ then either $Q(x)\equiv x^{2}$ or $Q(x)\equiv x^{2}+x$ or $Q(x)\equiv x^{2}+1$ or $Q(x)\equiv x^{2}+x+1$. None of the $Q(x)$'s does divide $P \pmod{2}$. We conclude that $2P$ cannot be  written in the form \eqref{eq.1.24} with deg$Q<3$, proving our claim. So, the proof of lemma is completed.       
\end{proof}
\section{Exceptional values for $k$}
Consider the equation \eqref{eq.1.3} for fixed $k\in\{1,3\}$ and fixed $n=m>1$. Then the equation  \eqref{eq.1.3} is equivalent to the equation
\begin{align}\label{eq.1.25}
(k+1)y^{m}=P(x)
\end{align}
where $P(x)=B_{q}(lx+1)-B_{q}(x+1)$, $q\in\{2,4\}$, $q=k+1.$

If $q=2$, then the equation \eqref{eq.1.25} becomes 
\begin{align} \label{eq.1.28}
2y^{m}=(l-1)x((l+1)x+1).
\end{align}
By using Lemma \ref{lem.2}, we have $r_{1}=r_{2}=1$ and so $t_{1}=t_{2}=1$. From here we get $m=2$. In the case $m=2$, the equation \eqref{eq.1.28} becomes
\begin{align}\label{eq.1.26}
u^{2}-2(l-1)v^{2}=1
\end{align} 
where $u=2x(l+1)+1$, $v=2(l+1)y$, $l \equiv 0 \pmod{2}$. By theory of Pell's equation (see e.g. \cite[Ch.8]{Mord}), for infinitely many choices of $l$, \eqref{eq.1.26} has infinitely many solutions.

If $q=4$, then the equation \eqref{eq.1.25} becomes
\begin{align} \label{eq.1.29}
4y^{m}=x^{2}(l-1)((l^{2}+1)x+l+1)((l+1)x+1).
\end{align}
Similarly to the former case, by Lemma \ref{lem.2} we get $m=2$. In this case, the equation \eqref{eq.1.29} becomes
\begin{align}\label{eq.1.27}
u^{2}-(l^{4}-1)v^{2}=-l^{2}(l+1)(l^{2}+1)(l-1)^{3}
\end{align}  
where $u=(l^{4}-1)t$ ($t\in\mathbb{Z}$), $v=(l^{4}-1)x+(l^{3}-1)$, $l \equiv 0 \pmod{2}$. So, \eqref{eq.1.27} has infinitely many solutions.
\begin{remark}
Even if $l$ is odd, the equations \eqref{eq.1.26} and \eqref{eq.1.27} are Pell's equations. But in this work, we consider the title equation where $l$ is even.
\end{remark}
\section{Proof of the theorems}
\vspace{3mm}
\textsc{Proof of Theorem \ref{theo.1}.}
Let $x,y\geq1$ and $n\geq2$ be an arbitrary solution of \eqref{eq.1.3} in integers.We know from Lemma \ref{lem.4} that $G(x)$ has at least two distinct zeroes. Hence by applying Lemma \ref{lem.3} it follows from the equation \eqref{eq.1.3}  that we get an effective bound for $n.$\vspace{3mm}

\textsc{Proof of Theorem \ref{theo.2}.}
We know from Theorem \ref{theo.1} that $n$ is bounded, i.e. $n<C_{1}$ with an effectively computable $C_{1}$. So we may assume that $n$ is fixed. Then we get the following equation in integers $x,y\geq1$
\begin{align*}
P(x)=y^{n}
\end{align*}
where $P$ is given by \eqref{eq.1.4} with $q=k+1$. Write
\begin{align*}
P(x)=a_{0}\displaystyle\prod_{i=1}^{n}(x-x_{i})^{r_{i}}
\end{align*}
where $a_{0}\neq 0$, $x_{i}\neq x_{j}$ if $i\neq j$ and, for a fixed $n$ let $t_{i}=\frac{n}{(n,r_{i})}.$ If $n$ is even, then by Lemma \ref{lem.6} at least three zeroes have odd multiplicity, say $r_{1},r_{2},r_{3}$. Hence $t_{1}$, $t_{2}$ and $t_{3}$ are even. Consequently the exceptional cases in Lemma \ref{lem.2} cannot occur. If $n$ is odd and $p|n$ for an odd prime $p$, then by Lemma \ref{lem.6} at least two zeroes of $P(x)$ have multiplicities prime    to $p$. We may assume that $(r_{1},p)=(r_{2},p)=1$, so $p|t_{1}$ and $p|t_{2}$. Using Lemma \ref{lem.2}, we have $\max\{x,y\}<C_{2}(n)$ with an effectively computable $C_{2}(n)$. Finally $n<C_{1}$ implies the required assertion. This proves the theorem.

\section*{Acknowledgements}
I would like to thank Professor \'{A}kos Pint\'er for his useful remarks and guidance and also I would like to cordially thank to all people in t​he Institute of Mathematics, University of Debrecen for ​their ​hospitality​. Finally I would like to thank referees for valuable comments. This work ​was supported by the Scientific and Technical Research Council of Turkey (T\"{U}B\.{I}TAK) ​under ​2219-International Postdoctoral Research Scholarship.

\end{document}